\documentclass[AMA,STIX1COL]{WileyNJD-v2}
\usepackage{ourpackages}
\usepackage{tabstackengine}

\def\nochapnum{\par
  \setcounter{chapter}{0}
  \setcounter{section}{0}
  \def\@chapapp{}
}

\def\chapnum{\par
  \setcounter{chapter}{0}
  \setcounter{section}{0}
  \def\@chapapp{Chapter}
}

\newcommand{\factbox}[3]
{
  \mbox{ }\\
  \centerline{%
    \framebox[0.9\textwidth]{%
      \ifthenelse
      {\equal{#2}{}}
      {\parbox{0.8\columnwidth}{\begin{fact}\label{#1} \mbox{ } \\{#3}\end{fact}}}
      {\parbox{0.8\columnwidth}{\begin{fact}[#2]\label{#1}\mbox{ }\\ {#3} \end{fact}}}
    }
    \mbox{ }\\
  }
}

\newcommand{\mysymb}[2]{
  \noindent
  \begin{tabular}[t]{p{2cm}@{\hspace{1em}:\hspace{1em}}p{9cm}}
    \makebox[2cm][r]{$#1$} & #2\end{ta}\\
}

\newcommand{\bd}{\begin{displaymath}}
    \newcommand{\ed}{\end{displaymath}}

\newcommand{\weg}[1]{}

\newcommand{\identm}[1]{\boldsymbol{\mathit{I}}_{#1}}

\newcommand{\zerom}[1]{\boldsymbol{\mathit{0}}_{#1}}

\newcommand{\htrans}[2]%
{\vphantom{T}_{#2}^{\hspace*{0.1ex}#1}%
{\hspace*{-0.1ex}\boldsymbol{T}}\vphantom{T}}
\newcommand{\expnumber}[2]{{#1}\mathrm{e}{#2}}

\newcommand{\dotT}[2]%
{\vphantom{T}_{#2}^{\hspace*{0.1ex}#1}%
{\dot{\hspace*{-0.2ex}\boldsymbol{T}}}\vphantom{T}}

\newcommand{\ddotT}[2]%
{\vphantom{T}_{#2}^{\hspace*{0.1ex}#1}%
{\ddot{\hspace*{-0.2ex}\boldsymbol{T}}}\vphantom{T}}

\newcommand{\infhtrans}[2]%
{\vphantom{T}_{#2}^{\hspace*{0.1ex}#1}%
{\hspace*{-0.3ex}\boldsymbol{T}}_{\hspace*{-0.8ex}\Delta}\vphantom{T}}

\newcommand{\esthtrans}[2]%
{\sideset{_{#2}^{\hspace*{0.1ex}#1}}{}{\vphantom{T}}%
  {\hspace*{-0.2ex}\mathop{\widehat{\boldsymbol{T}}}}\vphantom{T}}

\newcommand{\rot}[2]%
{\vphantom{R}_{#2}^{\hspace*{0.1ex}#1}%
{\hspace*{-0.2ex}\boldsymbol{R}}\vphantom{R}}

\newcommand{\dotR}[2]%
{\vphantom{R}_{#2}^{\hspace*{0.1ex}#1}%
{\dot{\hspace*{-0.2ex}\boldsymbol{R}}}\vphantom{R}}

\newcommand{\strans}[2]%
{\vphantom{S}_{#2}^{\hspace*{0.2ex}#1}%
{\hspace*{-0.3ex}\boldsymbol{S}}\vphantom{S}}

\newcommand{\dotS} [2]%
{\vphantom{S}_{#2}^{\hspace*{0.3ex}#1}%
{\dot{\hspace*{-0.3ex}\boldsymbol{S}}}\vphantom{S}}

\newcommand{\infstrans}[2]%
{\vphantom{S}_{#2}^{\hspace*{0.2ex}#1}%
{\hspace*{-0.3ex}\boldsymbol{S}}_{\hspace*{-0.3ex}\Delta}\vphantom{S}}

\newcommand{\sproj}[2]%
{\vphantom{P}_{#2}^{\hspace*{0.2ex}#1}%
{\hspace*{-0.3ex}\boldsymbol{P}}\vphantom{P}}

\newcommand{\dotP} [2]%
{\vphantom{P}_{#2}^{\hspace*{0.3ex}#1}%
{\dot{\hspace*{-0.3ex}\boldsymbol{P}}}\vphantom{P}}

\newcommand{\ptrans}[2]%
{\vphantom{M}_{#2}^{\hspace*{0.2ex}#1}%
{\hspace*{-0.3ex}\boldsymbol{M}}\vphantom{M}}

\newcommand{\dotM}[2]%
{\dot{\hspace*{-0.4ex}\boldsymbol{M}}\vphantom{M}^{#2}_{\hspace*{-0.3ex}#1}}

\newcommand{\infptrans}%
{\boldsymbol{M}_{\hspace*{-0.5ex}\Delta}}

\newcount\m \newcount\n
\def\hours{\n=\time \divide\n 60
  \m=-\n \multiply\m 60 \advance\m \time
  \twodigits\n\ :\ \twodigits\m}
\def\twodigits#1{\ifnum #1<10 0\fi \number#1}

\def\realnumbers{\mathbb{R}}
\newcommand{\realvectors}[1]{\mathbb{R}^{#1}}
\newcommand{\realmatrices}[2]{\mathbb{R}^{#1\times#2}}
\newcommand{\symmetricmatrices}[1]{\mathbb{S}^{#1}}
\newcommand{\posdefsymmetricmatrices}[1]{\mathbb{S}^{#1}_{++}}
\newcommand{\possemdefsymmetricmatrices}[1]{\mathbb{S}^{#1}_{+}}
\newcommand{\lowertriangularmatrices}[1]{\mathbb{L}^{#1}}
\newcommand{\uppertriangularmatrices}[1]{\mathbb{U}^{#1}}

\newcommand{\posdefdiagonalmatrices}[1]{\mathbb{D}^{#1}_{++}}
\newcommand{\permutationmatrices}[1]{\mathbb{P}^{#1}}
\newcommand{\orthogonalmatrices}[1]{\mathbb{Q}^{#1}}
\newcommand{\transpose}{^\prime}
\newcommand{\mintranspose}{^{\prime^{-1}}} 

\newcommand{\N}{T}
\newcommand{\pkg}[1]{\textsc{#1}}

\DeclareMathOperator*{\minimize}{minimize}

\DeclareRobustCommand{\cpluspluslogo}{\hbox{C\hspace{-0.5ex}
                       \protect\raisebox{0.5ex}
                       {\protect\scalebox{0.67}{++}}}}

\newcommand{\orcidm}[1]{\href{https://orcid.org/#1}{\textcolor[HTML]{A6CE39}{\includegraphics[keepaspectratio,width=0.7em]{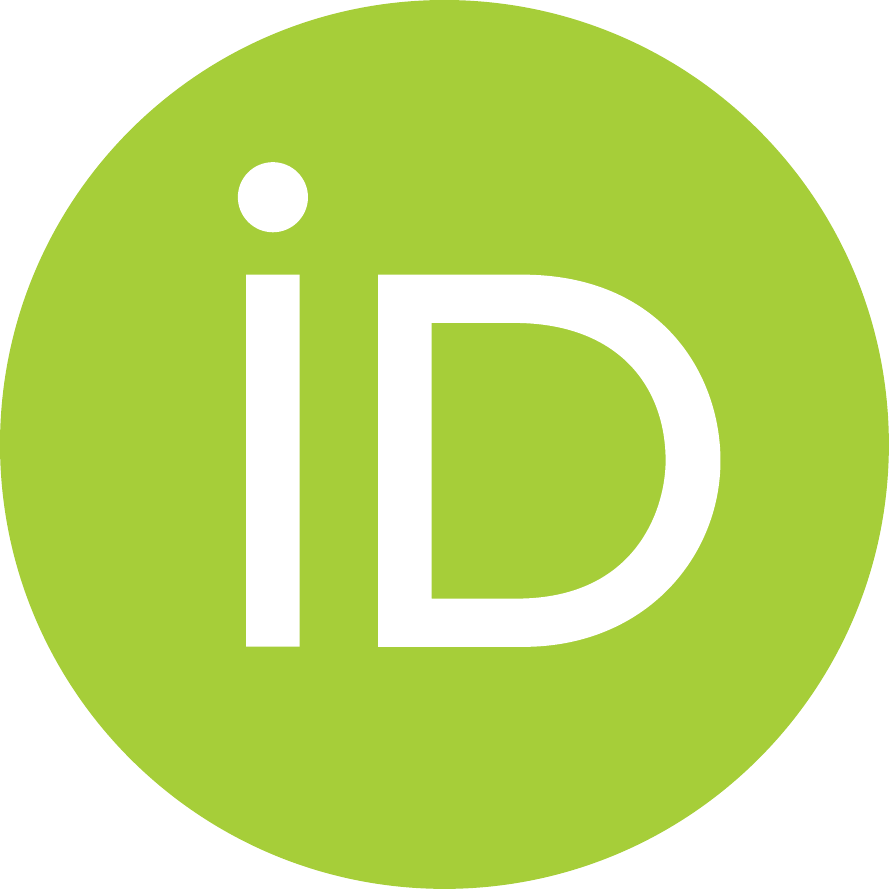}}} \href{https://orcid.org/#1}{https://orcid.org/#1}}
\hypersetup{
	colorlinks=true,
	urlcolor=blue, 
	citecolor=blue,
	linkcolor=blue,
	backref=true
}

\articletype{Research Article}%

\received{day month year}
\revised{day month year}
\accepted{day month year}

\begin{document}

\title{A Generalization of the Riccati Recursion for Equality-Constrained Linear Quadratic Optimal Control}

\author[]{Lander Vanroye}
\author[]{Joris De Schutter}
\author[]{Wilm Decré}
\authormark{Vanroye \textsc{et al.}}

\address[]{\orgdiv{Department of Mechanical Engineering}, \orgname{KU Leuven}, \orgaddress{\country{Belgium}}}
\address[]{\orgdiv{Core Lab ROB}, \orgname{Flanders Make@KU Leuven}, \orgaddress{\country{Belgium}}}

\corres{Lander Vanroye, KU Leuven, Department of Mechanical Engineering, Celestijnenlaan 300 box 2420, 3001 Heverlee, Belgium \email{lander.vanroye@kuleuven.be}}
\fundingInfo{European Research Council (ERC) under the European Union's Horizon 2020 research and innovation programme (Grant agreement: ROBOTGENSKILL No. 788298) and FWO project G0D1119N of the Research Foundation - Flanders (FWO - Flanders)}

\abstract[Summary]{
        This paper introduces a generalization of the well-known Riccati recursion for solving the discrete-time equality-constrained linear quadratic optimal control problem.
        The recursion can be used to compute problem solutions as well as optimal feedback control policies.
        Unlike other tailored approaches for this problem class, the proposed method does not require restrictive regularity conditions on the problem.
        This allows its use in nonlinear optimal control problem solvers that use exact Lagrangian Hessian information.
        We demonstrate that our approach can be implemented in a highly efficient algorithm that scales linearly with the horizon length.
        Numerical tests show a significant speed-up of about one order of magnitude with respect to state-of-the-art general-purpose sparse linear solvers.
        Based on the proposed approach, faster nonlinear optimal control problem solvers can be developed that are suitable for more complex applications or for implementations on low-cost or low-power computational platforms.
        The implementation of the proposed algorithm is made available as open-source software.
}

\maketitle

\section{Introduction}
The equality-constrained linear quadratic optimal control problem, or constrained LQ problem for short, is a generalization of the classical LQ problem \cite{rawlings2017model} that supports affine stagewise mixed input-state equality constraints.
Here, stagewise means that the constraints only relate inputs and states of a specific stage, or time step, of the control horizon.
An important type of these constraints are initial and terminal stage constraints, resulting in a two-point boundary value optimization problem.
Constrained LQ problems are encountered in linear quadratic optimal control applications but also as a subproblem in many nonlinear optimal control problem solvers.
In these nonlinear optimization solvers, at each iteration, the problem is approximated by a constrained LQ problem, resulting in a series of constrained LQ subproblems to be solved.
Two classes of nonlinear optimal control solver algorithms can be distinguished.
The first class are algorithms based on optimal feedback control policies. This class includes Differential Dynamic Programming (DDP) and the related iterative Linear Quadratic Regulator (iLQR).
Given a local linear approximation of system dynamics and a linear quadratic approximation of the problem Lagrangian, these algorithms compute an optimal feedback control policy at each stage in the backward pass.
These calculated feedback control policies are then used to compute the next iterate in the forward pass.
This process is repeated until a local minimum of the cost function is found.
Iterative LQR is a DDP-variant that uses a Gauss-Newton approximation of the Lagrangian Hessian.
Hence, it can be seen as if the second-order information of the dynamics and stagewise equality constraints are neglected \cite{giftthaler2018family}.
The second class of algorithms are Newton-type algorithms. 
These algorithms directly optimize for the open-loop trajectories of the linearized problem model to compute the next iterate.
These algorithms include Direct Single Shooting and Direct Multiple Shooting.
The latter is known to have superior convergence properties over the former \cite{albersmeyer2010lifted} and is able to be initialized from dynamically infeasible estimates of the solution.
Stagewise inequality constraints have been successfully incorporated by using penalty methods and barrier (interior point) methods that preserve the unconstrained LQ problem structure \cite{rao1998mpc_interiorpoint}.
This means that the classical Riccati recursion can be used without adaptation.
To date one of the main shortcomings of nonlinear optimal control problem solvers is the treatment of stagewise equality constraints.
Penalty methods have been used as well for equality constraints, but these methods are inexact or require more iterations than imposing the constraints directly.
Efficiently solving the constrained LQ problem is of particular importance because it often appears as the most time-consuming step of nonlinear optimal control solver algorithms.
Since the constrained LQ problem is a special case of an equality-constrained quadratic programming (QP) problem,
its solution can be found using the problem's Karush-Kuhn-Tucker (KKT) optimality conditions.
These result in a symmetric and indefinite linear system that can be solved using general-purpose linear solvers.
Dense linear solvers do not exploit the particular stagewise structure of the problem, and their computational complexity grows cubically with the control horizon length.
Sparse linear solvers, on the contrary, exploit the problem structure and scale linearly with the horizon length.
While general-purpose sparse linear solvers do exploit the block-sparse structure of the considered linear system, they do not allow computationally efficient implementations.

Tailored methods for solving the constrained LQ problem have been developed in prior work.
Domahidi et al. developed a structure-exploiting range space method that is able to solve a slightly more general problem formulation than the constrained LQ problem \cite{domahidi2012efficient}.
Sideris \& Rodriguez proposed a factorization scheme that can cope with stage-wise input-state equality constraints by eliminating the corresponding dual variables \cite{sideris2010riccati}.
Giffthaler \& Buchli showed that projecting the input variables onto the space of inputs that are admissible  with respect to the stagewise equality constraints results in a possibly singular optimal control problem. This optimal control problem can then be solved by a modified version of the classical Riccati recursion \cite{giftthaler2018family}.
Laine \& Tomlin applied a scheme where the nullspace of the input part of the stagewise equality constraints is used to modify the Riccati recurison \cite{laine2018efficient}.
A shortcoming of all these approaches is that they impose restrictive regularity conditions on the problem, such as a positive definite Lagrangian Hessian.
These regularity conditions inhibit their use in many applications, for example in nonlinear optimization algorithms make use of exact Lagrangian Hessian information.
In Section \ref{sec:discussion} we compare our proposed approach to these approaches in detail.

The salient features of our recursion and the contributions of this paper are fourfold:
\textit{First, we present a generalization of the Riccati recursion that allows stagewise equality constraints.}
This recursive factorization scheme allows for efficient implementations for both optimal feedback control policy and optimal solution computation.
\textit{Second, we avoid imposing restrictive regularity conditions on the problem.} For example, unlike several other methods  \cite{sideris2010riccati,giftthaler2017projection,domahidi2012efficient} we do not impose positive definiteness of the full-space Hessian,
such that our approach is not restricted to Lagrangian Hessian approximations that are guaranteed to be positive definite, and is hence much more flexible in the choice of using exact or approximate Lagrangian Hessian information, in a nonlinear programming algorithm context.
\textit{Third, our recursion detects ill-posed problems, i.e. problems that do not have a unique minimizer, without additional computational cost.}
In line-search nonlinear programming algorithms it is common to detect such ill-posed (sub)problems because they do not guarantee a descent direction with respect to the current iterate.
State-of-the-art nonlinear programming solvers such as \pkg{Ipopt} \cite{wachter2006implementation} modify the full-space Lagrangian Hessian with an inertia correction while \pkg{Knitro} \cite{knitro} switches to a trust-region step in such case.
\textit{Fourth, we present the recursion step by step from a linear algebra perspective, with a detailed mathematical analysis and discussion.}
Instead of approaching the problem from a dynamic programming perspective, we look at it as a stagewise factorization of the full KKT system.
This allows us to elegantly prove the claims we make for the second and third contribution using only linear algebra.

The structure of this paper is as follows.
First, Section \ref{sec:prere} discusses the notation and some prerequisites that are necessary throughout the rest of the paper.
Then, Section \ref{sec:probdef} defines the considered constrained LQ problem class.
Subsequently, Section \ref{sec:recursion} gives an overview of the generalized Riccati recursion in a step-by-step fashion.
Section \ref{sec:implementation} discusses some practical implementation aspects and provides the full algorithm for solving the constrained LQ problem.
Section \ref{Sec:numres} compares the performance and accuracy of the proposed method with three general-purpose sparse linear solvers on a set of randomly generated problems and a real-world quadrotor problem.
Section \ref{sec:discussion} compares the proposed recursion with related work, how the approach can also be used to derive optimal feedback control policies and the method's limitations.
Finally, Section \ref{sec:conclusions} provides the conclusions and outlook.
\pagebreak

\section{Notation and background} \label{sec:prere}
\subsection{Notation}
Table \ref{table:notation} summarizes the notation used throughout the paper.
\begin{table}[h]
        \centering
        \begin{tabular}{l|l}
                $\realvectors{n}$                & set of real column vectors of size $n$                             \\
                $\realmatrices{m}{n}$            & set of $m \times n$ real matrices                                  \\
                $\symmetricmatrices{n}$          & set of $n \times n$ real symmetric matrices                        \\
                $\posdefsymmetricmatrices{n}$    & set of $n \times n$ positive definite real symmetric matrices      \\
                $\possemdefsymmetricmatrices{n}$ & set of $n \times n$ positive semi-definite real symmetric matrices \\
                $\permutationmatrices{n}$        & set of $n \times n$ permutation matrices                           \\
                $\orthogonalmatrices{n}$         & set of $n \times n$ orthogonal matrices                            \\
                $\lowertriangularmatrices{n}$    & set of $n \times n$ real lower triangular matrices                 \\
                $\uppertriangularmatrices{n}$    & set of $n \times n$ real upper triangular matrices                 \\
                $\posdefdiagonalmatrices{n}$     & set of $n \times n$ positive definite real diagonal matrices       \\
                $\identm{n}$                     & $n \times n$ identity matrix                                       \\
                $\zerom{m \times n}$             & $m \times n$ zero matrix                                           \\
                $A \transpose$                   & transpose of $A$                                                   \\
        \end{tabular}
        \caption{\label{table:notation} notation used throughout the paper}
\end{table}

\subsection{KKT system and reduced Hessian} \label{sec:redhess}
Consider the equality-constrained quadratic programming (QP) problem:
\begin{subequations} \label{eq:generalqp}
        \begin{align}
                \minimize_{\substack{\mathbf{x}}} \quad & \frac{1}{2}\mathbf{x}\transpose H \mathbf{x} + h\transpose \mathbf{x} \\
                \textrm{subject to} \quad      & A\mathbf{x} + \mathbf{a} = \mathbf{0}  ,
        \end{align}
\end{subequations}
with primal optimization variables $x \in \realvectors{n}$, Hessian $H \in \symmetricmatrices{n}$, and constraint Jacobian $A \in \realmatrices{m}{n}$, with $m \leq n$.
The first-order necessary conditions for $\mathbf{x}^*$ to be a solution of \eqref{eq:generalqp}, state that there is a vector $\mathbf{\lambda}^* \in \realvectors{m}$ such that:
\begin{equation} \label{eq:kdef}
        \underbrace{\begin{bmatrix} H & A\transpose \\ A & \end{bmatrix}}_{\mathcal{K}} \begin{bmatrix} \mathbf{x}^* \\ \mathbf{\lambda}^* \end{bmatrix} = \begin{bmatrix}- \mathbf{h} \\- \mathbf{a} \end{bmatrix},
\end{equation}
in which $\mathcal{K}$ is referred to as the Karush-Kuhn-Tucker (KKT) matrix of problem \eqref{eq:generalqp} \cite{nocedal2006numerical}. 
We refer to the set of equations from the top and bottom block row as, respectively, the stationarity and the constraint equations.
In the remainder of the paper we use the following more compact notation for systems of equations like \eqref{eq:kdef}:
\begin{equation}
 \begin{bNiceArray}{cc|c}[margin,parallelize-diags=false, first-row]
                x & \lambda     &   \\
                H & A\transpose & \mathbf{h} \\
                A &             & \mathbf{a} \\
        \end{bNiceArray}
        ,
\end{equation}
in which we omit the asterisks ($*$).
\begin{definition}[Reduced Hessian of a KKT matrix] \label{def:reducedhess}
        A reduced Hessian $R$ is equal to the Hessian $H$, projected on the nullspace of
        the constraint Jacobian $A$: \[ R = Z\transpose H Z ,\] where $Z$ is defined as a matrix whose columns span the nullspace of
        $A$. 
        In case the nullspace of $A$ is empty, the reduced Hessian is defined as an empty matrix.
\end{definition}
Note that in the definition above, we speak about a (and not the) reduced Hessian, since reduced Hessians are nonunique as they depend on the choice of the nullspace basis used in $Z$.
In the remainder of this paper we will omit this detail and speak about the reduced Hessian because it is less confusing and common practice in optimization textbooks.
\begin{theorem}[Existence of a unique minimizer]
        Assuming full row rank of the constraint Jacobian,
        problem \eqref{eq:generalqp} has a unique minimizer if and only if the reduced Hessian of its KKT matrix is positive definite \cite{nocedal2006numerical}.
\end{theorem}
\subsection{Decomposition}\label{sec:decomp}
\begin{theorem}
        Any real matrix $A \in \mathbb{R}^{m\times n}$, with $m$ and $n$ nonzero and $\rho$ the rank of $A$, can be decomposed as:
\end{theorem}
\begin{equation}\label{eq:decomp}
        A = \N_L \begin{bmatrix}-\identm{\rho} &  \\  & \zerom{(m-\rho) \times(n- \rho)}\end{bmatrix} \N_R,
\end{equation}
with $\N_L \in \mathbb{R}^{m\times m}$ and $\N_R \in \mathbb{R}^{n\times n}$ invertible matrices.
\begin{proof}
        Proof by construction.
        There are several options for constructing such decomposition:
        \begin{itemize}
                \item{\textbf{Singular Value Decomposition (SVD)} \cite{matrixcomputations}}

                Performing an SVD of $A$ yields:
                \begin{equation*}
                        A = \underbrace{\begin{bmatrix} U_P & U_N \end{bmatrix}}_{U} \begin{bmatrix} \Sigma & \\& \zerom{(m-\rho) \times (n-\rho)}\end{bmatrix} V\transpose
                \end{equation*}
                with $U \in \orthogonalmatrices{m}$, $V \in \orthogonalmatrices{n}$ and $\Sigma \in \posdefdiagonalmatrices{\rho}$.
                We can obtain a decomposition of the form \eqref{eq:decomp} by choosing $\N_L$ and $\N_R$ as:
                \\
                \begin{equation*}
                        A = \underbrace{\begin{bmatrix} -U_P\Sigma & U_N \end{bmatrix}}_{\N_L} \begin{bmatrix} -\identm{\rho} &  \\ & \zerom{(m-\rho) \times (n-\rho)}\end{bmatrix}
                        \underbrace{V\transpose}_{\N_R}.
                \end{equation*}

                \item{\textbf{LU decomposition with complete pivoting} \cite{matrixcomputations}}

                Performing an LU decomposition of $A$ yields:
                \begin{equation*}
                        {A = P_L\transpose \underbrace{\begin{bmatrix} L_1 &  \\ L_2  & \identm{\scriptscriptstyle m-\rho} \end{bmatrix}}_{L}  \underbrace{\begin{bmatrix} U_1  & U_2 \\ \zerom{\scriptscriptstyle(m-\rho)\times \rho} & \zerom{\scriptscriptstyle (m-\rho)\times(n-\rho)} \end{bmatrix}}_{U} P_R},
                \end{equation*}
                with $P_L \in \permutationmatrices{m}$,
                $P_R \in \permutationmatrices{n}$ permutation matrices,
                $L_1 \in \lowertriangularmatrices{\rho}$, an invertible lower triangular matrix, $U_1 \in \uppertriangularmatrices{\rho}$ an invertible upper triangular matrix,
                $L_2 \in \realmatrices{(m-\rho)}{\rho}$ and $U_2 \in \realmatrices{\rho}{(n-\rho)}$.
                The LU decomposition can be used to factor $A$ as:
                \\
                \begin{equation*}
                        {A = \underbrace{P_L\transpose L\begin{bmatrix} -U_1  &  \\   & \identm{\scriptscriptstyle m-\rho} \end{bmatrix}}_{\N_L}  \begin{bmatrix}  -\identm{\rho} & \\ & \zerom{(m-\rho)\times(n-\rho)} \end{bmatrix}\underbrace{\begin{bmatrix}  \identm{\rho} & { U_1^{-1}U_2} \\  & \identm{n-\rho} \end{bmatrix} P_R}_{\N_R}.}
                \end{equation*}

                \item{\textbf{Other decompositions}}

                Other decompositions that can be used include QR with column pivoting and the Complete Orthogonal decomposition.
                A discussion on these decompositions can be found in Golub et al. \cite{matrixcomputations}.
        \end{itemize}
\end{proof}

In the remainder of this paper, if $m=0$ or $n=0$, we define $T_L = \identm{m}$ and $T_R = \identm{n}$, with $\identm{0}$ empty.
Furthermore, we omit all linear algebra operations involving empty matrices and we define the product of two matrices $A \in \mathbb{R}^{m \times 0}$, and $B \in \mathbb{R}^{0 \times n}$ as $AB = \zerom{m\times n}$.

\section{Problem definition}\label{sec:probdef}
This paper considers the constrained LQ problem written in the following standard form:
\begin{subequations}
        \begin{align}
                \minimize_{\substack{\mathbf{x}_k, \mathbf{u}_k,\mathbf{x}_K}} \quad & l_K(\mathbf{x}_K) + \sum_{k=0}^{K-1} l_k(\mathbf{u}_k, \mathbf{x}_k) \label{eq:LQOCP-objective}         \\
                \textrm{subject to} \quad                 & \mathbf{x}_{k+1} = B_k \mathbf{u}_k +A_k \mathbf{x}_k +  \mathbf{b}_k \label{eq:LQOCP-dynamicsconstraints}       \\
                                                          & G_K \mathbf{x}_{K}  =  - \mathbf{g}_K   \label{eq:LQOCP-endconstraint}                         \\
                                                          & G_{k,u} \mathbf{u}_k +G_{k,x}\mathbf{x}_k   = - \mathbf{g}_k         \label{eq:LQOCP-stageconstraints},
        \end{align}
        \label{eq:LQOCP}
\end{subequations}
for $k=0,\ldots,K-1$ in \eqref{eq:LQOCP-dynamicsconstraints} and \eqref{eq:LQOCP-stageconstraints}, and with $l_K(\mathbf{x}_K)$ and $l_k(\mathbf{u}_k, \mathbf{x}_k)$ defined as:
\begin{align*}
         & l_K(x_K) =  \frac{1}{2} \mathbf{x}_K\transpose Q_K \mathbf{x}_K + \mathbf{q}_K\transpose \mathbf{x}_K                                                                                                                                                                                                                  \\
         & l_k(\mathbf{u}_k, \mathbf{x}_k) = \frac{1}{2} \begin{bmatrix} \mathbf{u}_k \\ \mathbf{x}_k \end{bmatrix}\transpose \begin{bmatrix} R_k & S_k\transpose \\ S_k & Q_k\end{bmatrix}\begin{bmatrix} \mathbf{u}_k \\ \mathbf{x}_k \end{bmatrix} + \begin{bmatrix} \mathbf{r}_k \\ \mathbf{q}_k \end{bmatrix}\transpose\begin{bmatrix} \mathbf{u}_k \\ \mathbf{x}_k \end{bmatrix},
\end{align*}
with state vector $\mathbf{x}_k \in \realnumbers^{n_x}$,
input vector $\mathbf{u}_k \in \realnumbers^{n_u}$,
$R_k \in \symmetricmatrices{n_u}$, $S_k \in \realmatrices{n_x}{n_u}$, $Q_k \in \symmetricmatrices{n_x}$ and $K$ the horizon length, or number of control intervals.
The problem has stagewise quadratic objective function terms $l_k$,
and affine system dynamics \eqref{eq:LQOCP-dynamicsconstraints}, terminal-state \eqref{eq:LQOCP-endconstraint} and stagewise mixed state-input constraints \eqref{eq:LQOCP-stageconstraints}.
We refer to Section \ref{sec:recursion}, Equation \eqref{eq:originalKKT} for the KKT system structure of this problem.
Apart from linear independence of the constraints \eqref{eq:LQOCP-dynamicsconstraints}-\eqref{eq:LQOCP-stageconstraints},
the only assumption that is made on the problem's regularity,  is that it is well-posed, meaning that it has a unique minimizer.
As explained in Section \ref{sec:redhess}, this is equivalent to the requirement that the reduced Hessian of the KKT matrix of the problem is positive definite.
\section{Approach} \label{sec:recursion}
In this section the recursive scheme for solving the KKT system of the constrained LQ problem is outlined. For notational brevity, we consider $K=2$ in the main text, but in the algorithm environments we always give the general case (i.e. for any $K$). 
The KKT system of problem \eqref{eq:LQOCP} is given by:

\begin{equation}
        \label{eq:originalKKT}
 \begin{bNiceArray}{cccccccccc|c}[margin,parallelize-diags=false, first-row]
                \mathbf{x}_2        & \mathbf{v}_2           & \boldsymbol{\pi}_2         & \mathbf{u}_1     & \mathbf{x}_1           & \boldsymbol{\lambda}_1         & \boldsymbol{\pi}_1         & \mathbf{u}_0     & \mathbf{x}_0           & \boldsymbol{\lambda}_0         &     \\
                Q_2        & G_2\transpose & -\identm{}    &         &               &                   &               &         &               &                   & \mathbf{q}_2 \\
                G_2        &               &               &         &               &                   &               &         &               &                   & \mathbf{g}_2 \\
                -\identm{} &               &               & B_1     & A_1           &                   &               &         &               &                   & \mathbf{b}_1 \\
                           &               & B_1\transpose & R_1     & S_1\transpose & G_{1,u}\transpose &               &         &               &                   & \mathbf{r}_1 \\
                           &               & A_1\transpose & S_1     & Q_1           & G_{1,x}\transpose & -\identm{}    &         &               &                   & \mathbf{q}_1 \\
                           &               &               & G_{1,u} & G_{1,x}       &                   &               &         &               &                   & \mathbf{g}_1 \\
                           &               &               &         & -\identm{}    &                   &               & B_0     & A_0           &                   & \mathbf{b}_0 \\
                           &               &               &         &               &                   & B_0\transpose & R_0     & S_0\transpose & G_{0,u}\transpose & \mathbf{r}_0 \\
                           &               &               &         &               &                   & A_0\transpose & S_0     & Q_0           & G_{0,x}\transpose & \mathbf{q}_0 \\
                           &               &               &         &               &                   &               & G_{0,u} & G_{0,x}       &                   & \mathbf{g}_0 \\
        \end{bNiceArray}
        .
\end{equation}
\\
Here, $\mathbf{v}_K$, $\boldsymbol{\lambda}_k$ and $\boldsymbol{\pi}_k$ are the dual optimization variables associated with
the terminal-state constraint \eqref{eq:LQOCP-endconstraint},
the mixed state-input constraints \eqref{eq:LQOCP-stageconstraints},
and the dynamics constraints \eqref{eq:LQOCP-dynamicsconstraints}, respectively.
We exploit the block-banded diagonal structure of the KKT matrix through a Riccati-inspired backward substitution followed by a forward substitution \cite{gianluca_efficient}.
At each step of the backward substitution, a subset of the primal and dual optimization variables are eliminated, starting from the terminal stage.
In the forward substitution the primal and dual variables are computed in the reverse order of the backward substitution, resulting in a bottom-up algorithm, starting from the initial stage. 

As described below, the order of eliminating variables in the backward substitution is chosen in such a way that the top left submatrix of the transformed KKT matrix at the end of the stagewise factorization has the same structure as the KKT matrix at the beginning.
As a result, the same steps can be applied recursively for the following, prior in time, stage.
This procedure can be repeated until the initial stage is reached, as described in the final paragraph of this section.

\noindent \textbf{Substitution of dynamics.}\label{sec:subsdyn}
Eliminate $\boldsymbol{\pi}_2$ and $\mathbf{x}_2$ from the first and third block row, respectively. Substitution of $\mathbf{x}_2 = B_1 \mathbf{u}_1 + A_1 \mathbf{x}_1 + \mathbf{b}_1$ and $\boldsymbol{\pi}_2 = Q_2 \mathbf{x}_2 +G_2\transpose \mathbf{v}_2 + \mathbf{q}_2 $ results in a transformed KKT system with top left matrix:
\begin{equation}\label{eq:KKT_after_subsdyn}
 \begin{bNiceArray}{ccc|c}[first-row]
                \mathbf{u}_1                & \mathbf{x}_1                      & \left(\mathbf{v}_2  \quad \boldsymbol{\lambda}_1 \right) &                \\
                \overline{R}_1     & \overline{S}_1\transpose & \overline{G}_{1,u}\transpose                               & \overline{\mathbf{r}}_1 \\
                \overline{S}_1     & \overline{Q}_1           & \overline{G}_{1,x}\transpose                               & \overline{\mathbf{q}}_1 \\
                \overline{G}_{1,u} & \overline{G}_{1,x}       &   \ddots                                                         & \overline{\mathbf{g}}_1 \\
\end{bNiceArray}.
\end{equation}
The matrix $\begin{bmatrix}\overline{G}_{k,u}&\overline{G}_{k,x} &\overline{\mathbf{g}}_k \end{bmatrix}$ is a concatenation of the mixed state-input constraint of stage $k$ (here $k=1$) and the state constraint on $\mathbf{x}_{k+1}$, after substitution of the dynamics, while
$\gamma_k$ is the number of equations represented by this matrix.

\noindent \textbf{Symmetric transformation.} 
Using the decomposition explained in Section \ref{sec:decomp}, decompose $\overline{G}_{1,u}$ as:
\begin{equation*} \label{eq:factormiddle}
        \overline{G}_{1,u}= \N_{1,L} \left[\begin{array}{cc} -\identm{\rho_1} &\\  &\zerom{(\gamma_1 - \rho_1) \times (n_u-\rho_1)}  \end{array}\right] \N_{1,R},
\end{equation*}
with $\rho_1$ the rank and $\gamma_1$ the number of rows of $\overline{G}_{1,u}$. 
Now define $\widetilde{\mathbf{u}}_{1,A}, \widetilde{\boldsymbol{\lambda}}_1 \in \mathbb{R}^{\rho_1}$,  and $\widetilde{\mathbf{u}}_{1,B}, \mathbf{v}_1 \in \mathbb{R}^{\gamma_1-\rho_1}$ and apply the following transformation of variables:
\begin{equation}
        \left\{
        \begin{array}{l}
                \mathbf{u}_1 = \N_{1,R}^{-1} \begin{pmatrix}
                                            \widetilde{\mathbf{u}}_{1,A} \\ \widetilde{\mathbf{u}}_{1,B}
                                    \end{pmatrix} \\
                \begin{pmatrix}
                        \mathbf{v}_2 \\ {\boldsymbol{\lambda}}_1
                \end{pmatrix}
                =\N_{1,L}\mintranspose
                \begin{pmatrix}
                        \widetilde{\boldsymbol{\lambda}}_1 \\ \mathbf{v}_1
                \end{pmatrix}
        \end{array}
        \right.,
\end{equation}
and left-multiply the first and the third block row of \eqref{eq:KKT_after_subsdyn} by $\N_{R,1}\mintranspose$ and $\N_{L,1}^{-1}$, respectively. The top left block of the KKT system now becomes:
\begin{equation}\label{eq:KKT_after_transformation}
 \begin{bNiceArray}{cccc|c}[margin,parallelize-diags=false, first-row, last-col]
                \left(\widetilde{\mathbf{u}}_{1,A} \quad \widetilde{\mathbf{u}}_{1,B} \right) & \mathbf{x}_1                       & \widetilde{\boldsymbol{\lambda}_1}                                      & \mathbf{v}_1           &                 &                                      \\
                \widetilde{R}_1                                                                    & \widetilde{S}_1\transpose & \begin{bmatrix} -\identm{\rho_1} \\ \zerom{} \end{bmatrix} &               & \widetilde{\mathbf{r}}_1 &                                      \\
                \widetilde{S}_1                                                                    & \widetilde{Q}_1           & \widetilde{G}_{1,x}\transpose                              & H_1\transpose & \widetilde{\mathbf{q}}_1 &                                      \\
                \begin{bmatrix}-\identm{\rho_1}& \zerom{}\end{bmatrix}                             & \widetilde{G}_{1,x}       &                                                            &               & \widetilde{\mathbf{g}}_1 & (\ref{eq:KKT_after_transformation}a) \\
                                                                                                   & H_1                       &                                                            &   \ddots            & {\mathbf{h}}_1           & (\ref{eq:KKT_after_transformation}b) \\
\end{bNiceArray}.
\end{equation}
Note that the stagewise constraints are recombined into two subsets.
The top subset (\ref{eq:KKT_after_transformation}a) can always be satisfied by choosing the appropriate input $\widetilde{\mathbf{u}}_{1,A}$.
The bottom subset (\ref{eq:KKT_after_transformation}b) represents a constraint that cannot be satisfied by choosing appropriate inputs at the current stage: it puts a constraint on the state $\mathbf{x}_1$.
It can be seen as if the transformation of this step transfers this subset of the constaints to the next, prior in time, stage.
The columns of $T_{1,R}^{-1}\begin{bmatrix} \zerom{(n_u -\rho_1)\times\rho_1} & I_{n_u - \rho_1} \end{bmatrix}\transpose$ form a basis for the nullspace of $\overline{G}_{1,u}$
and, hence, $\widetilde{\mathbf{u}}_{1,B}$ can be seen as a parameterization of the linear subspace of inputs that are admissible without violating the stagewise constraints.
The size of $\widetilde{\mathbf{u}}_{1,A}$ and $\widetilde{\mathbf{u}}_{1,B}$ is hence equal to the rank and nullity of $\overline{G}_{1,u}$, respectively.
We refer to Table \ref{table:dims} for the dimensions of all submatrices involved in the algorithm.

\noindent \textbf{Substitution of stagewise equality constraints.}
Eliminate $\widetilde{\mathbf{u}}_{1,A}$ and $\widetilde{\boldsymbol{\lambda}}_1$
by substituting $\widetilde{\mathbf{u}}_{1,A} = \widetilde{G}_{1,x} \mathbf{x}_1 + \widetilde{\mathbf{g}}_1$ and $\begin{pmatrix}\widetilde{\boldsymbol{\lambda}}_1\transpose & - \end{pmatrix}\transpose = \widetilde{R}_1 \begin{pmatrix}
                        \widetilde{\mathbf{u}}_{\scriptstyle {1,A}}\transpose & \widetilde{\mathbf{u}}_{1,B} \transpose
                \end{pmatrix}\transpose + \widetilde{S}_1\transpose \mathbf{x}_1 + \widetilde{\mathbf{r}}_1$, the KKT system is transformed to a system with top left matrix:
\begin{equation}\label{eq:KKT_hat}
 \begin{bNiceArray}{ccc|c}[margin,parallelize-diags=false, first-row]
                \widetilde{\mathbf{u}}_{1,B} & \mathbf{x}_1                     & \mathbf{v}_1           &               \\
                \widehat{R}_1       & \widehat{S}_1\transpose &               & \widehat{\mathbf{r}}_1 \\
                \widehat{S}_1       & \widehat{Q}_1           & H_1\transpose & \widehat{\mathbf{q}}_1 \\
                                    & H_1                     & \ddots        & \mathbf{h}_1           \\
\end{bNiceArray}.
\end{equation}
Note that decision variable $\widetilde{\mathbf{u}}_{1,B}$ only appears in the Hessian, and not in the constraint Jacobian of the (transformed) KKT system \eqref{eq:KKT_hat}.
This means that these variables are unconstrained. Since, essentially, we are still solving the original optimization problem \eqref{eq:LQOCP}
it is easy to see that, if the original problem has a unique solution, the optimization problem associated with KKT system \eqref{eq:KKT_hat} should have a unique solution as well.
This means that $\widehat{R}_1$ should be positive definite.
This fact is used in the next step. For a formal proof we refer to the appendix of this dissertation.

\noindent \textbf{Schur complement step.}
Eliminate $\widetilde{\mathbf{u}}_{1,B}$, by using the Schur complement of $\widehat{R}_1$, resulting in the system:

\begin{equation}\label{eq:KKT_P}
 \begin{bNiceArray}{cccccc|c}[margin,parallelize-diags=false, first-row]
                \mathbf{x}_1        & \mathbf{v}_1           & \boldsymbol{\pi}_1         & \mathbf{u}_0     & \mathbf{x}_0           & \boldsymbol{\lambda}_0         &     \\
                P_1        & H_1\transpose & -\identm{}    &         &               &                   & \mathbf{p}_1 \\
                H_1        &               &               &         &               &                   & \mathbf{h}_1 \\
                -\identm{} &               &               & B_0     & A_0           &                   & \mathbf{b}_0 \\
                           &               & B_0\transpose & R_0     & S_0\transpose & G_{0,u}\transpose & \mathbf{r}_0 \\
                           &               & A_1\transpose & S_0     & Q_0           & G_{0,x}\transpose & \mathbf{q}_0 \\
                           &               &               & G_{0,u} & G_{0,x}       &                   & \mathbf{g}_0 \\
        \end{bNiceArray}
        .
\end{equation}
Because $\widehat{R}_1$ is positive definite, the Cholesky decomposition $\widehat{R}_1 = \Lambda_1 \Lambda_1\transpose$, with $\Lambda_1 \in \mathbb{L}^{n_u-\rho}$ and invertible, can be used to efficiently calculate this Schur complement.
If this Cholesky decomposition fails, it means that the problem is ill-posed.
This closes the recursion, since the top left submatrix of the transformed KKT matrix now has the same form as the top left submatrix of the original KKT matrix.
This means that the previous steps can be applied recursively until all primal and dual optimization variables except for the initial stage are eliminated.

\noindent \textbf{Factorization of initial stage.}
When all primal optimization variables except for $x_0$ are eliminated, the KKT system is transformed into the system:
\begin{equation}
 \begin{bNiceArray}{cc|c}[margin,parallelize-diags=false, first-row]
                \mathbf{x}_0 & \mathbf{v}_0           &     \\
                P_0 & H_0\transpose & \mathbf{p}_0 \\
                H_0 &               & \mathbf{h}_0 \\
        \end{bNiceArray}.
\end{equation}
Now we factor the matrix $H_0$ into:
\begin{equation}\label{eq:factorinit} H_0 = \N_{I,L} \begin{bmatrix} -\identm{\rho_I} & \zerom{} \end{bmatrix} \N_{I,R} .\end{equation}
Note that $H_0$ has full row rank because linear independence of the constraints is assumed.
Hence, there are no zero rows in the middle matrix of the decomposition, and $\rho_I$ is equal to the number of rows of $H_0$: $\rho_I = \gamma_0 - \rho_0$.
The $I$-subscript indicates that the quantities are associated with the initial stage.
This decomposition gives rise to a transformation of variables:
\begin{equation}
        \left\{ \begin{array}{l}
                \mathbf{x}_0 = \N_{I,R}^{-1} \begin{pmatrix} \widetilde{\mathbf{x}}_{0,A} \\ \widetilde{\mathbf{x}}_{0,B} \end{pmatrix} \\
                \mathbf{v}_0 = \N_{I,L}\mintranspose\widetilde{\mathbf{v}}_0
        \end{array} \right. ,
\end{equation}
with $\mathbf{x}_{0,A} \in \mathbb{R}^{\rho_I}$ and $\mathbf{x}_{0,B} \in \mathbb{R}^{n_x - \rho_I}$.
After multiplying the first block row with $\N_{I,R}\mintranspose$ and the last block row with $\N_{I,L}^{-1}$, the system is transformed to:
\begin{equation}
        \begin{bNiceArray}{cc|c}[first-row]
                \left( \widetilde{\mathbf{x}}_{0, A} \quad \widetilde{\mathbf{x}}_{0, B} \right) & \widetilde{\mathbf{v}}_0         &       \\
                \widetilde{P}_I                                                                         & \begin{bmatrix}
                                                                                                                  -I_{\rho_I} \\ \zerom{}
                                                                                                          \end{bmatrix} & \widetilde{\mathbf{p}}_I \\
                \begin{bmatrix}
                        -I_{\rho_I} & \zerom{}
                \end{bmatrix}                                                                 &                         & \widetilde{\mathbf{h}}_I \\
        \end{bNiceArray}.
\end{equation}

Now eliminate $\widetilde{\mathbf{x}}_{0,A} =\widetilde{\mathbf{h}}_I$ and $\begin{pmatrix} \widetilde{\mathbf{v}}_0\transpose & \mathbf{-} \end{pmatrix}\transpose =  \widetilde{\mathbf{p}}_0 \begin{pmatrix} \widetilde{\mathbf{x}}_{0,A}\transpose & \widetilde{\mathbf{x}}_{0,B}\transpose\end{pmatrix}\transpose + \widetilde{P}_I$. This results in the system:
\begin{equation}
        \widehat{P}_I \widetilde{\mathbf{x}}_{I,B} = - \widehat{\mathbf{p}}_I.
\end{equation}
Following the same reasoning as before, $\widehat{P}_I$ has to be positive definite, so the system can be solved efficiently using the Cholesky decomposition of $\widehat{P}_I = \Lambda_I \Lambda_I\transpose$, with $\Lambda_I \in \mathbb{L}^{n_x-\rho_I}$ and invertible.
If this Cholesky decomposition would fail, it means that the problem is ill-posed.

\begin{table}[]
        \centering
        \begin{tabular}{c|cc}
                                                                             & rows                    & columns        \\ \hline
                $A_k$                                                        & $n_x$                   & $n_x$          \\
                $B_k$                                                        & $n_x$                   & $n_u$          \\
                $R_k$, $\overline{R}$, $\widetilde{R}_k$                     & $n_u$                   & $n_u$          \\
                $S_k$, $\overline{S}$, $\widetilde{S}_k$                     & $n_x$                   & $n_u$          \\
                $Q_k$, $\overline{Q}$, $\widetilde{Q}_k, \widehat{Q}_k, P_k$ & $n_x$                   & $n_x$          \\
                $\overline{G}_u$                                             & $\gamma_k$              & $n_u$          \\
                $\overline{G}_x$                                             & $\gamma_k$              & $n_x$          \\
                $\N_{k,L}$                                                   & $\gamma_k$              & $\gamma_k$     \\
                $\N_{k,R}$                                                   & $n_u$                   & $n_u$          \\
                $H_k$                                                        & $\gamma_{k} - \rho_{k}$ & $n_x$          \\
                $E_k$                                                        & $\gamma_k - \rho_{k}$ & $n_u - \rho_k$          \\
                $Z_k$                                                        & $n_u -\rho_k$ &  $\gamma_k - \rho_{k}$        \\
                $\widehat{R}_k, \Lambda_k$                                   & $n_u -\rho_k$           & $n_u - \rho_k$ \\
                $\widehat{S}_k$                                              & $n_x$                   & $n_u -\rho_k$  \\
                $\widetilde{G}_{k,x}$                                        & $\rho_k $     & $n_x$          \\
                $L_k$                                                        & $n_u -\rho_k$           & $n_x$          \\
                $\N_{I,L}$                                                   & $\gamma_0$              & $\gamma_0$     \\
                $\N_{I,R}$                                                   & $n_x$                   & $n_x$          \\
                $\widetilde{H}_I$                                            & $\gamma_0 - \rho_I$     & $n_x-\rho_I$   \\
                $\widetilde{P}_I$, $\widehat{P}_I, \Lambda_I$                & $n_x - \rho_I$          & $n_x-\rho_I$   \\
        \end{tabular}
        \caption{\label{table:dims} Dimensions of submatrices appearing in the algorithm}
\end{table}

\section{Implementation} \label{sec:implementation}
\noindent \textbf{Matrix decomposition of numerically rank-deficient matrices.} \label{sec:decomprd}
Numerical issues can occur in the symmetric transformation step of the backward substitution (Algorithm \ref{alg:fullalgo}, Line \ref{alg:linetransvar}).
In this situation the matrix decomposition, as explained in  Section \ref{sec:decomp}, can cause problems if $\overline{G}_{u}$ is numerically rank-deficient.
This can occur even if the constraint Jacobian of the full KKT system has full rank. 
For clarity, we repeat the decomposition \eqref{eq:decomp}, applied to $\overline{G}_{u}$, at stage $k$ in the backward substitution:
\begin{equation} \label{eq:decomp1}
        \overline{G}_u = \N_{k,L} \begin{bmatrix}-\identm{\rho_k} &  \\  & \zerom{(m-\rho_k) \times(n- \rho_k)}\end{bmatrix} \N_{k, R}.
\end{equation}
When $\overline{G}_u$ is nearly rank-deficient, during the factorization a decomposition of the following form can be encountered:
\begin{equation} \label{eq:decomp2}
        \overline{G}_u = \N_{k,L} \begin{bmatrix}-\identm{\rho_k} &  \\  & E_{k, (m-\rho) \times(n- \rho)}\end{bmatrix} \N_{k,R},
\end{equation}
where $E_k$ represents a matrix consisting of very small values ($E_{i,j} < \epsilon$).
When the decomposition is obtained by use of an SVD factorization, $E$ corresponds to the diagonal matrix consisting of the singular values smaller than the threshold $\epsilon$.
In the case of an LU factorization with complete pivoting, this type of matrix decomposition can be found by early stopping the factorization when no pivot larger than $\epsilon$ can be found.
Rather than enforcing the decomposition of the form \eqref{eq:decomp1}, from a numerical point of view, it is better to continue with the decomposition of the form \eqref{eq:decomp2} to avoid using very small pivots.
The modifications of the algorithm due to this different decomposition are described below.
The top left KKT system \eqref{eq:KKT_after_transformation} changes to:
\begin{equation}\label{eq:KKT_after_transformation_modificiation}
 \begin{bNiceArray}{cccc|c}[margin,parallelize-diags=false, first-row]
                \left(\widetilde{\mathbf{u}}_{1,A} \; \widetilde{\mathbf{u}}_{1,B}\right) & \mathbf{x}_1                       & \widetilde{\boldsymbol{\lambda}_1}                                      & \mathbf{v}_1                                                                  \\
                \widetilde{R}_1                                                                    & \widetilde{S}_1\transpose & \begin{bmatrix} -\identm{\rho_1} \\ \zerom{} \end{bmatrix} & \begin{bmatrix} \zerom{} \\ E_1 \transpose \end{bmatrix}              & \widetilde{\mathbf{r}}_1                                       \\
                \widetilde{S}_1                                                                    & \widetilde{Q}_1           & \widetilde{G}_{1,x}\transpose                              & H_1\transpose & \widetilde{\mathbf{q}}_1                                       \\
                \begin{bmatrix}-\identm{\rho_1}& \zerom{}\end{bmatrix}                             & \widetilde{G}_{1,x}       &                                                            &               & \widetilde{\mathbf{g}}_1 \\
                \begin{bmatrix}\zerom{}& E_1\end{bmatrix}                            & H_1                       &                                                            & \ddots             & {\mathbf{h}}_1            \\
        \end{bNiceArray}.
\end{equation}
As before, substitute $\widetilde{\mathbf{u}}_{1,A} = \widetilde{G}_{1,x} \mathbf{x}_1 + \widetilde{\mathbf{g}}_1$ and $\begin{pmatrix}\widetilde{\boldsymbol{\lambda}}_1\transpose & - \end{pmatrix}\transpose = \widetilde{R}_1 \begin{pmatrix}
                        \widetilde{\mathbf{u}}_{1,A}\transpose & \widetilde{\mathbf{u}}_{1,B}
                \end{pmatrix}\transpose + \widetilde{S}_1\transpose \mathbf{x}_1 + \widetilde{\mathbf{r}}_1$,
                which results in a top left KKT system:
\begin{equation}\label{eq:KKT_hat_modificiation}
 \begin{bNiceArray}{ccc|c}[margin,parallelize-diags=false, first-row]
                \widetilde{\mathbf{u}}_{1,B} & \mathbf{x}_1                     & \mathbf{v}_1           &       \\
                \widehat{R}_1       & \widehat{S}_1\transpose &  E_1 \transpose     & \widehat{\mathbf{r}}_1 \\
                \widehat{S}_1       & \widehat{Q}_1           & H_1\transpose & \widehat{\mathbf{q}}_1 \\
                   E_1              & H_1                     &     \ddots          & \mathbf{h}_1           \\
 \end{bNiceArray}.
\end{equation}
Now, as before, eliminate $\widetilde{\mathbf{u}}_{1,B}$, by using the Schur complement of $\widehat{R}_1$, i.e. $\widetilde{u}_{1,B} = -\widehat{R}_1 ^{-1} \left(\widehat{S}_1 \transpose \mathbf{x}_1 + E_1 \transpose \mathbf{v}_1 + \widehat{\mathbf{r}}_1 \right)$, resulting in a system of the form:
\begin{equation}\label{eq:KKT_P_modificiation}
 \begin{bNiceArray}{cccccc|c}[margin,parallelize-diags=false, first-row]
                \mathbf{x}_1        & \mathbf{v}_1           & \boldsymbol{\pi}_1         & \mathbf{u}_0     & \mathbf{x}_0           & \boldsymbol{\lambda}_0         &     \\
                P_1        & H_1\transpose & -\identm{}    &         &               &                   & \mathbf{p}_1 \\
                H_1        &  -E_1 \hat{R_1} ^{-1} E_1 \transpose   &               &         &               &                   & \mathbf{h}_1 \\
                -\identm{} &               &               & B_0     & A_0           &                   & \mathbf{b}_0 \\
                           &               & B_0\transpose & R_0     & S_0\transpose & G_{0,u}\transpose & \mathbf{r}_0 \\
                           &               & A_1\transpose & S_0     & Q_0           & G_{0,x}\transpose & \mathbf{q}_0 \\
                           &               &               & G_{0,u} & G_{0,x}       &                   & \mathbf{g}_0 \\
 \end{bNiceArray}.
\end{equation}
Note that the form of this matrix only differs from \eqref{eq:KKT_P} by the presence of the entry  $-E_1\hat{R_1} ^{-1} E_1 \transpose$.
As the values of this entry are very small, it can be neglected in the numerical computations, and the computation continues as before.
The modifications, due to this different decomposition, are taken into account in the final backward and forward algorithm (Algorithm \ref{alg:fullalgo} and \ref{alg:forward}, respectively).

\noindent \textbf{Iterative refinement.} Iterative refinement is a common method for improving the accuracy of a linear system solution \cite{duff2004ma57, mumps}.
The algorithm improves the accuracy of the solution by solving a sequence of linear systems, where the right hand side of the system is the residual of the previous solution.
The residual is defined as the difference between the left hand side evaluated at the solution and the right hand side of the system ($\mathbf{r} = A\mathbf{x} + \mathbf{b}$).
Since the coefficient matrix of the involved linear systems is unchanged, the factorization of the coefficient matrix from the backward substitution can be reused.

\noindent \textbf{Full algorithm.}
Algorithm \ref{alg:fullalgo} outlines the full backward substitution algorithm for arbitrary horizon length $K$, while Algorithm \ref{alg:forward} outlines the forward substitution algorithm.
Note that some of the matrices appearing in these algorithms are structurally zero, identity or symmetric.
In our software implementation this structure is exploited in order to avoid unnecessary floating point operations and memory usage.
The quantities without a stage index $k$ are temporary variables, that are not required later in the algorithm.
\begin{algorithm}[ht] \caption{{Backward Substitution}} \label{alg:fullalgo}
        \normalsize
        \begin{algorithmic}[1]
                \State $\begin{bmatrix} H_K & \mathbf{h}_K  \end{bmatrix}  \gets \begin{bmatrix} G_K & \mathbf{g}_K \end{bmatrix}$
                \State $\begin{bmatrix} P_K & p_K  \end{bmatrix}  \gets \begin{bmatrix} Q_K & \mathbf{q}_K \end{bmatrix}$
                \For{\texttt{k = K-1, \dots, 0}}
                \State  ${\begin{bmatrix} \overline{R} & \overline{S}\transpose &\overline{\mathbf{r}} \\ \overline{S} & \overline{Q} & \overline{\mathbf{q}} \end{bmatrix}\gets
                                        \begin{bmatrix} \begin{matrix} B\transpose_k \\ A\transpose_k \end{matrix} & \identm{n_u+n_x} \end{bmatrix}\begin{bmatrix} P_{k+1} & & & \mathbf{p}_{k+1} \\ & R_k &S_k\transpose & \mathbf{r}_k \\ & S_k & Q_k & \mathbf{q}_k \end{bmatrix} \begin{bmatrix} \begin{matrix} B_k & A_k & \mathbf{b}_k \end{matrix}\\ \identm{n_u+n_x+1}  \end{bmatrix}}$
                \Comment substitution of dynamics
                \State  ${\begin{bmatrix} \overline{G}_u &\overline{G}_x  &\overline{\mathbf{g}} \end{bmatrix}
                                        \gets \begin{bmatrix} \begin{bmatrix} G_{k, u} & G_{k, x} & \mathbf{g}_k \end{bmatrix}  \\ \begin{bmatrix} H_{k+1} & \mathbf{h}_{k+1} \end{bmatrix} \begin{bmatrix} \begin{matrix} B_k & A_k\end{matrix} & \mathbf{b}_k \\  & 1  \end{bmatrix} \end{bmatrix}}$
                \State $\gamma_k \gets$ number of rows of $\overline{G}_u$
                \State  Calculate $\N_{k,L}, \; \N_{k,R},\; \rho_k$ and $E_k$ such that $\overline{G}_u$ can be decomposed as in Section \ref{sec:decomprd}. \label{alg:linetransvar}
                \Comment symmetric transformation 
                \State  ${\begin{bmatrix} \widetilde{R}_k & \widetilde{S}\transpose_k &\widetilde{\mathbf{r}}_k \\ \widetilde{S}_k & \widetilde{Q}_k & \widetilde{\mathbf{q}}_k \end{bmatrix}\gets \begin{bmatrix} \N_{k,R} \mintranspose& \\ & \identm{n_x} \end{bmatrix} \begin{bmatrix} \overline{R} & \overline{S}\transpose &\overline{\mathbf{r}} \\ \overline{S} & \overline{Q} & \overline{\mathbf{q}} \end{bmatrix}\begin{bmatrix} \N_{k,R}^{-1} & \\ & \identm{n_x+1} \end{bmatrix}}$
                \State ${\begin{bmatrix}\begin{matrix}\widetilde{G}_{k,x} & \widetilde{\mathbf{g}}_k \\  H_k & \mathbf{h}_k \end{matrix} \end{bmatrix} \gets \N_{k,L} ^{-1}  \begin{bmatrix}\overline{G}_x&\overline{\mathbf{g}} \end{bmatrix}}$
                \Comment substitution of equality constraints 
                \State  ${\begin{bmatrix} \widehat{R}_k & \widehat{S}_k\transpose &\widehat{\mathbf{r}}_k \\ \widehat{S}_k & \widehat{Q}_k & \widehat{\mathbf{q}}_k \end{bmatrix}\gets \begin{bmatrix} \begin{matrix} \zerom{} \\ \widetilde{G}_{k,x}\transpose \end{matrix} & \identm{n_u+n_x-\rho_k} \end{bmatrix} \begin{bmatrix} \widetilde{R}_k & \widetilde{S}\transpose_k &\widetilde{\mathbf{r}}_k \\ \widetilde{S}_k & \widetilde{Q}_k & \widetilde{\mathbf{q}}_k \end{bmatrix} \begin{bmatrix} \begin{matrix} \zerom{}& \widetilde{G}_{k,x} & \widetilde{\mathbf{g}}_{k}\end{matrix} \\ \identm{n_u+n_x-\rho_k+1}\end{bmatrix}}$
                \State $\Lambda_k \gets \mbox{chol}\left(\widehat{R}_k\right)$
                \Comment Schur complement step 
                \State $\begin{bmatrix} L_k & Z_k & \mathbf{l}_k \end{bmatrix} \gets \Lambda_k^{-1}\begin{bmatrix} \widehat{S}\transpose_k & E_k \transpose &\widehat{\mathbf{r}}_k\end{bmatrix} $
                \State $\begin{bmatrix} P_k & \mathbf{p}_k \end{bmatrix} \gets \begin{bmatrix} \widehat{Q}_k & \widehat{\mathbf{q}}_k \end{bmatrix} - L_k\transpose\begin{bmatrix} L_k & \mathbf{l}_k \end{bmatrix}$
                \State $\begin{bmatrix} H_k & \mathbf{h}_k \end{bmatrix} \gets \begin{bmatrix} H_k & \mathbf{h}_k \end{bmatrix} - Z_k \transpose\begin{bmatrix} L_k & \mathbf{l}_k  \end{bmatrix}$
                \EndFor
                \State  Calculate $\N_{I,L}, \; \N_{I,R},\; \rho_I$, such that $H_0$ can be decomposed as in Section \ref{sec:decomp}.
                \Comment factorization of initial stage 
                \State $\begin{bmatrix} \widetilde{P}_I & \widetilde{\mathbf{p}}_I \end{bmatrix} \gets \N_{I,R} \mintranspose\begin{bmatrix} P_0 & \mathbf{p}_0 \end{bmatrix} \begin{bmatrix} \N_{I,R} ^{-1} & \\ & 1 \end{bmatrix} $
                \State $\widetilde{h}_I\gets \N_{I,L}^{-1}  \mathbf{h}_0 $
                \State ${\begin{bmatrix} \widehat{P}_I & \widehat{\mathbf{p}}_I \end{bmatrix} \gets \begin{bmatrix} \zerom{} & \identm{n_x-\rho_I} \end{bmatrix} \begin{bmatrix} \widetilde{P}_I & \mathbf{p}_I\end{bmatrix} \begin{bmatrix} \begin{matrix} \zerom{\rho_I \times (n_x-\rho_I)}& \widetilde{\mathbf{h}}_I \end{matrix} \\ \identm{n_x-\rho_I+1} \end{bmatrix}}$
                \State $\Lambda_I \gets \mbox{chol}\left(\widehat{P}_I\right)$
        \end{algorithmic}
\end{algorithm}
\begin{algorithm}[ht] \caption{{Forward substitution}} \label{alg:forward}
        \begin{algorithmic}[1]
                \State $\widetilde{\mathbf{x}}_{I,B} \gets - \Lambda_I \mintranspose\left(\Lambda_I^{-1} \widehat{\mathbf{p}}_I\right)$
                \State $\widetilde{\mathbf{x}}_{I} \gets \begin{bmatrix} \widetilde{\mathbf{h}}_I \\ \widetilde{\mathbf{x}}_{I,B} \end{bmatrix}$
                \State $\mathbf{x}_0 \gets \N_{I,R}^{-1} \widetilde{\mathbf{x}}_I$
                \State $ \begin{pmatrix} {\widetilde{\mathbf{v}}}_0 \\ \mathbf{-} \end{pmatrix} \gets \widetilde{P}_0 \widetilde{x}_I+ \widetilde{\mathbf{p}}_I$
                \State $ {\mathbf{v}}_0 \gets T_{I,L} \mintranspose\widetilde{\mathbf{v}}_0$
                \For{\texttt{k = 0, \dots, K-1}}
                \State  $\widetilde{\mathbf{u}}_{B} \gets -\Lambda_{k} \mintranspose\left(L_k  \mathbf{x}_k  + Z_k \mathbf{v}_k + \mathbf{l}_k  \right)$ \label{alg:fullalgo-utildeB}
                \State $\widetilde{\mathbf{u}} \gets  \begin{bmatrix} \widetilde{G}_{k,u}\mathbf{x}_k + \widetilde{\mathbf{g}}_{k} \\ \widetilde{\mathbf{u}}_{B} \end{bmatrix}$ \label{alg:fullalgo-utilde}
                \State $\mathbf{u}_k \gets \N_{k,R}^{-1} \widetilde{\mathbf{u}}$ \label{alg:fullalgo-u}
                \State $\begin{pmatrix}\widetilde{\boldsymbol{\lambda}}_k \\ - \end{pmatrix} \gets \widetilde{R}_k \widetilde{\mathbf{u}} +  \widetilde{S}_k\transpose \mathbf{x}_k + \widetilde{\mathbf{r}}_k$
                \State $
                        \begin{pmatrix}
                                \mathbf{v}_{k+1} \\ {\boldsymbol{\lambda}}_k
                        \end{pmatrix}
                        \gets \N_{k,L}\mintranspose
                        \begin{pmatrix}
                                \widetilde{\boldsymbol{\lambda}}_k \\ v_k
                        \end{pmatrix}$
                \State $\mathbf{x}_{k+1} \gets  B_k \mathbf{u}_k + A_k \mathbf{x}_k + \mathbf{b}_k$
                \State $\boldsymbol{\pi}_{k+1} \gets  P_{k+1} \mathbf{x}_{k+1} + H_{k+1}\transpose \mathbf{v}_{k+1} + \mathbf{p}_{k+1}$
                \EndFor
        \end{algorithmic}
\end{algorithm}

\noindent \textbf{Implementation details.}\label{sec:impl}
Algorithm \ref{alg:fullalgo} and Algorithm \ref{alg:forward} are implemented in efficient \cpluspluslogo-code.
The implementation is inspired by the implementation of the classical Riccati recursion in the QP-solver HPIPM \cite{gianluca_efficient, frison2020hpipm}.
As the recursion consists of a series of operations on (small-scale) stagewise matrices that fit in cache for typical problem sizes, we make use of the \pkg{Blasfeo} \cite{frison2018blasfeo} library, which is performance-optimized for this case.
We take into account the structural zero and identity submatrices when performing block matrix operations.
For the decompositions of type \eqref{eq:decomp} we use an LU factorization with complete pivoting (pivot treshold is fixed at $\epsilon = 1e-5$) due to its low computational cost and the presence of structure in $\N_L$ and $\N_R$, which we exploit in the implementation.

\section{Numerical results}\label{Sec:numres}
The performance of the proposed algorithm was benchmarked against three state-of-the-art general-purpose sparse linear solvers for
symmetric indefinite systems: \pkg{MA57} (version 3.11.1) \cite{duff2004ma57}, \pkg{MUMPS} (version 5.4.1) \cite{mumps} and \pkg{PARDISO} (version 6.0) \cite{pardiso72a} \cite{pardiso72b}\cite{pardiso72c}.
Iterative refinement was turned off for all sparse linear solvers.
For fairness of comparison, automatic scaling was turned off for \pkg{MA57}, as we found out that it slowed down the solution process significantly.
The packages were linked with \pkg{METIS} \cite{karypis1997metis} for constructing a fill-in reducing reordering in the analysis phase and \pkg{Intel MKL} as BLAS and LAPACK library.
For \pkg{MUMPS} we used the sequential version while \pkg{MA57} and \pkg{PARDISO} were configured to run in a single thread.
We compiled \pkg{MA57} and \pkg{MUMPS} with \pkg{gfortran}, version 9.3.0, with compiler optimization flag \texttt{-O3}.
The analysis phase and dynamic work array memory allocation were excluded from the timings.
The C/\cpluspluslogo-code of the implementation of the proposed algorithm and \pkg{BLASFEO} were compiled with \pkg{GCC}, version 9.3.0, with compiler optimization flag \texttt{-O3}.
We used \texttt{X64\_INTEL\_HASWELL} as the target for the \pkg{BLASFEO} library, such that it made use of the \texttt{AVX2} and \texttt{FMA} Intel\textregistered~ Instruction Set Extensions that were available on the CPU of the test machine.
Our test machine was a notebook computer equipped with an Intel\textregistered~ Core\texttrademark~ i7-10850H Processor, running Ubuntu 20.04. 
The CPU clock frequency was set fixed at 2.7 GHz.

The benchmark set consisted of randomly generated problems and a quadrotor problem.
For every problem we evaluated the performance and accuracy of the proposed algorithm and the three sparse linear solvers.
To evaluate the performance we used the wall evaluation time and to evaluate the accuracy we used the infinity norm residual ratio.
The residual norm ratio is defined as the ratio of the infinity norm of the residual vector to the infinity norm of the right-hand side vector ($\|A \mathbf{x}_\text{sol} + \mathbf{b}\|_\infty / \|\mathbf{b}\|_\infty$).
For the random problems, ten random constrained LQ's were generated for every test dimension, and we took the mean wall time and worst case infinity norm residual ratio as performance and accuracy measure, respectively.

\subsection{Random problems with initial and terminal constraints}
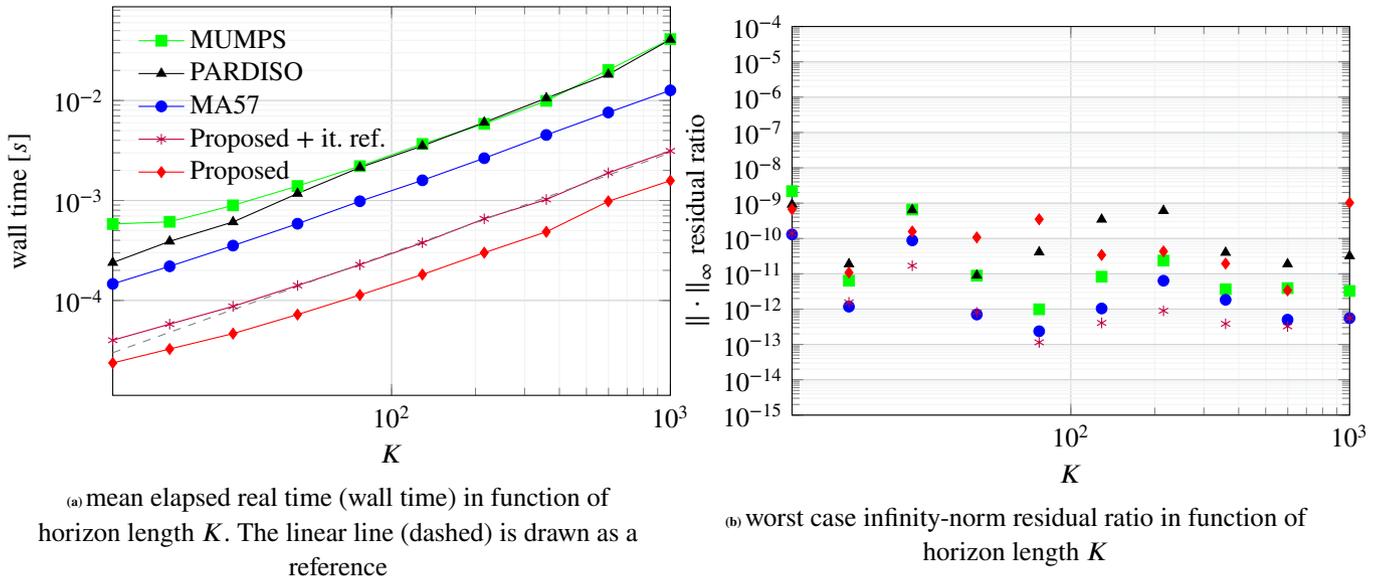
\begin{figure*}[h]
        \begin{subfigure}[t]{.5\textwidth}
                \vskip 0pt
                \captionsetup{width=.9\linewidth,justification=centering}
                \pgfplotstableread[col sep = comma]{output.csv}{\table}
                \begin{tikzpicture}
                        \begin{axis}[
                                xmode=log, ymode=log,
                                xmin = 10, xmax = 1000,
                                ymin = 0, ymax = 0,
                                xtick distance = 10,
                                ytick distance = 10,
                                grid = both,
                                minor tick num = 10,
                                major grid style = {lightgray},
                                minor grid style = {lightgray!25},
                                width = 1.0\linewidth,
                                height = 0.75\linewidth,
                                legend cell align = {left},
                                legend pos = north west,
                                legend style ={draw=none, fill=none},
                                xlabel = $K$,
                                ylabel = wall time \text{[}$s$\text{]}
                                        ]
                                \addplot[green, mark = square*] table [x = {log(hor)}, y = {log(mumps)}] {\table};
                                \addplot[black, mark = triangle*] table [x = {log(hor)}, y = {log(pardiso)}] {\table};
                                \addplot[blue, mark = *] table [x = {log(hor)}, y = {log(ma57)}] {\table};
                                \addplot[purple, mark = asterisk] table [x = {log(hor)}, y = {log(itref)}] {\table};
                                \addplot[red, mark = diamond*] table [x = {log(hor)}, y = {log(fatrop)}] {\table};
                                \addplot[gray, dashed] plot coordinates{(10,0.00003) (1000,0.003)};
                                \legend{
                                        MUMPS,
                                        PARDISO,
                                        MA57,
                                        Proposed + it. ref.,
                                        Proposed,
                                }
                        \end{axis}
                \end{tikzpicture}
                \caption{\normalsize mean elapsed real time (wall time) in function of horizon length $K$. The linear line (dashed) is drawn as a reference}\label{fig:aa}
        \end{subfigure}%
        \begin{subfigure}[t]{.5\textwidth}
                \vskip 0pt
                \captionsetup{width=.9\linewidth,justification=centering}
                \pgfplotstableread[col sep = comma]{output_acc_random1.csv}{\table}
                \begin{tikzpicture}
                        \begin{axis}[
                                xmode=log, ymode=log,
                                xmin = 10, xmax = 1000,
                                ymin = 1e-15, ymax = 0.0001,
                                xtick distance = 10,
                                ytick distance = 10,
                                grid = both,
                                minor tick num = 10,
                                major grid style = {lightgray},
                                minor grid style = {lightgray!25},
                                width = 1.0\linewidth,
                                height = 0.75\linewidth,
                                legend cell align = {left},
                                legend pos = south east,
                                legend style ={draw=none},
                                xlabel = $K$,
                                ylabel = $\|\cdot \|_\infty$ residual ratio 
                                        ]
                                \addplot[green, mark = square*, only marks] table [x = {log(hor)}, y = {log(mumps)}] {\table};
                                \addplot[blue, mark = *, only marks] table [x = {log(hor)}, y = {log(ma57)}] {\table};
                                \addplot[black, mark = triangle*, only marks] table [x = {log(hor)}, y = {log(pardiso)}] {\table};
                                \addplot[red, mark = diamond*, only marks] table [x = {log(hor)}, y = {log(fatrop)}] {\table};
                                \addplot[purple, mark = asterisk, only marks] table [x = {log(hor)}, y = {log(itref)}] {\table};
                        \end{axis}
                \end{tikzpicture}
                \caption{\normalsize worst case infinity-norm residual ratio in function of horizon length $K$}\label{fig:ab}
        \end{subfigure}
        \caption{\normalsize performance and numerical accuracy of proposed algorithm and general-purpose solvers for randomly generated problems in function of horizon length $K$ for problems with $n_x = 10$, $n_u = 5$ and ten constraints on the initial and terminal state.} \label{fig:random1}
\end{figure*}        
The first experiment consisted of randomly generated problems of fixed control input and state dimensions, $n_u=5$, $n_x=10$.
We fully constrained the initial and terminal state with $n_x$ constraints each. The horizon length $K$ was varied from $10$ to $1000$.
The results are shown in Figure \ref{fig:random1}.
For all solvers the measured wall times scaled linearly with the horizon length.
The proposed algorithm without iterative refinement was roughly eight times faster than \pkg{MA57}, the fastest general-purpose sparse linear solver for these problems.
The proposed algorithm with iterative refinement was roughly three to four times faster than \pkg{MA57}.
All problems reached an residual norm ratio lower than $10^{-9}$ without iterative refinement and lower than $10^{-10}$ after a single iterative refinement step.
The accuracy of the proposed algorithm was comparable to the general-purpose solvers.
Iterative refinement was able to improve the accuracy of all problems. 
\subsection{Random problems with constraints on every stage}
\begin{figure*}[h]
        \begin{subfigure}[t]{.5\textwidth}
                \vskip 0pt
                \captionsetup{width=.9\linewidth,justification=centering}
                \pgfplotstableread[col sep = comma]{output2.csv}{\table}
                \begin{tikzpicture}
                        \begin{axis}[
                                xmode=log, ymode=log,
                                xmin = 4, xmax = 180,
                                ymin = 0, ymax = 25,
                                xtick distance = 10,
                                ytick distance = 10,
                                grid = both,
                                minor tick num = 10,
                                major grid style = {lightgray},
                                minor grid style = {lightgray!25},
                                width = 1.0\linewidth,
                                height = 0.75\linewidth,
                                legend cell align = {left},
                                legend pos = north west,
                                legend style ={draw=none, fill=none},
                                xlabel = $n_x$,
                                ylabel = wall time \text{[}$s$\text{]}
                                        ]
                                \addplot[green, mark = square*] table [x = {log(hor)}, y = {log(mumps)}] {\table};
                                \addplot[black, mark = triangle*] table [x = {log(hor)}, y = {log(pardiso)}] {\table};
                                \addplot[blue, mark = *] table [x = {log(hor)}, y = {log(ma57)}] {\table};
                                \addplot[purple, mark = asterisk] table [x = {log(hor)}, y = {log(itref)}] {\table};
                                \addplot[red, mark = diamond*] table [x = {log(hor)}, y = {log(fatrop)}] {\table};
                                \addplot[gray, dashed] plot coordinates{(1,0.00000025) (1000,250)};

                                \legend{
                                        MUMPS,
                                        PARDISO,
                                        MA57,
                                        Proposed + it. ref.,
                                        Proposed,
                                }
                        \end{axis}
                \end{tikzpicture}
                \caption{\normalsize mean elapsed real time (wall time) in function of state dimension $n_x$. The cubic line (dashed) is drawn as a reference\\}\label{fig:ab}
        \end{subfigure}
        \begin{subfigure}[t]{.5\textwidth}
                \vskip 0pt
                \captionsetup{width=.9\linewidth,justification=centering}
                \pgfplotstableread[col sep = comma]{output_acc_random2.csv}{\table}
                \begin{tikzpicture}
                        \begin{axis}[
                                xmode=log, ymode=log,
                                xmin = 4, xmax = 180,
                                ymin = 1e-15, ymax = 0.0001,
                                xtick distance = 10,
                                ytick distance = 10,
                                grid = both,
                                minor tick num = 10,
                                major grid style = {lightgray},
                                minor grid style = {lightgray!25},
                                width = 1.0\linewidth,
                                height = 0.75\linewidth,
                                legend cell align = {left},
                                legend pos = south east,
                                legend style ={draw=none},
                                xlabel = $n_x$,
                                ylabel = $\|\cdot\|_\infty$ residual ratio 
                                        ]
                                \addplot[blue, mark = *, only marks] table [x = {log(hor)}, y = {log(ma57)}] {\table};
                                \addplot[green, mark = square*, only marks] table [x = {log(hor)}, y = {log(mumps)}] {\table};
                                \addplot[black, mark = triangle*, only marks] table [x = {log(hor)}, y = {log(pardiso)}] {\table};
                                \addplot[red, mark = diamond*, only marks] table [x = {log(hor)}, y = {log(fatrop)}] {\table};
                                \addplot[purple, mark = asterisk, only marks] table [x = {log(hor)}, y = {log(itref)}] {\table};
                        \end{axis}
                \end{tikzpicture}
                \caption{\normalsize worst case infinity-norm residual ratio in function of state dimension $n_x$\\}\label{fig:bb}
        \end{subfigure}
        \caption{performance and numerical accuracy of proposed algorithm and general-purpose linear solvers for randomly generated problems in function of state dimension $n_x$ for problems with $K =100$, $n_u = n_x/2$ and $n_x/4$ constraints on every state of the control horizon.} \label{fig:random2}
\end{figure*}
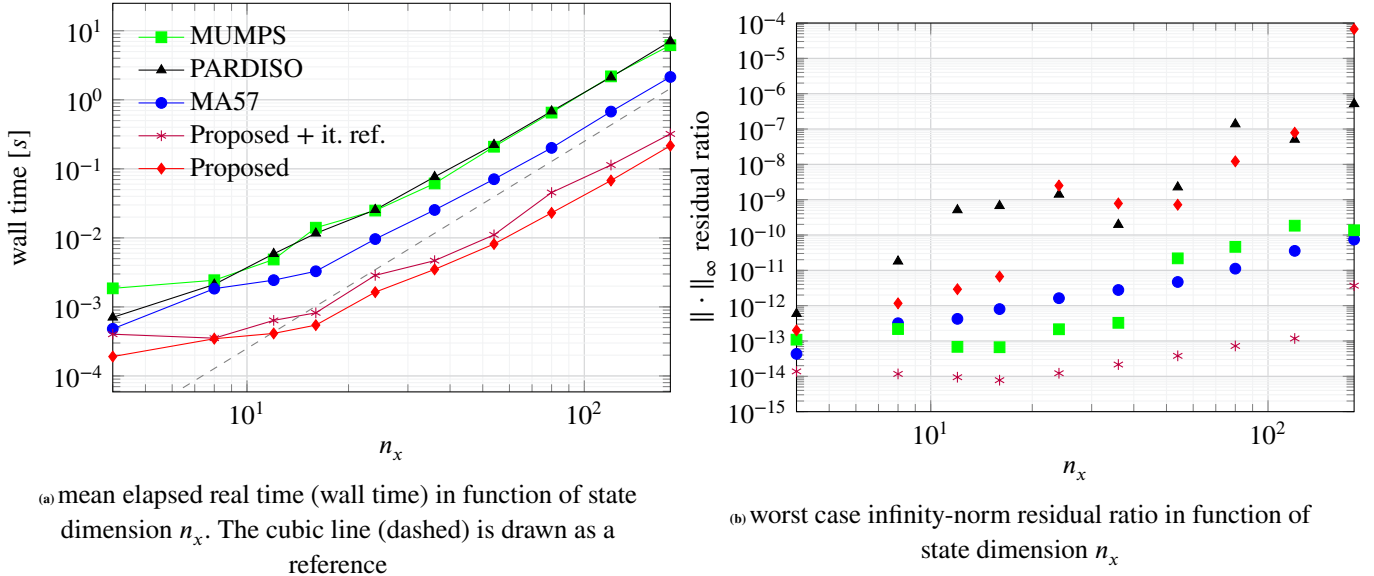
The second experiment consisted of randomly generated problems of fixed horizon length of $K=100$, the state dimension $n_x$ was varied between $4$ and $180$, the control input dimensions was $n_u=n_x/2$.
We put $n_x/4$ constraints on each state of the control horizon.
The results are shown in Figure \ref{fig:random2}.
For large enough problem dimensensions the measured wall times scaled cubically with the state dimension $n_x$.
The proposed algorithm without iterative refinement was roughly ten times faster than \pkg{MA57}, the fastest general-purpose sparse linear solver for these problems.
The proposed algorithm with iterative refinement was roughly four to five times faster than \pkg{MA57}.
The numerical accuracy of the proposed algorithm detoriated as problem dimensions increased.
However, iterative refinement was always able to improve the accuracy to a high precision.

\subsection{Quadrotor problem}
The state vector of this quadrotor problem consisted of the position, velocity, and three Euler angles, representing the quadrotor's orientation.
The control inputs were the acceleration in the upward direction of the drone and the Euler angle rates.
The dynamical model was based on the work of Zhang et al. \cite{Zhang2014}.
The initial equality constraint fixed the full state at the reference position and a disturbance in velocity and orientation from equilibrium, while the terminal constraint fixed the full state at the reference position in equilibrium.
The quadratic objective encoded the task to stabilize the quadrotor with minimum input energy.
The problem dimensions were $n_x = 10$, $n_u = 4$, $K = 50$, and $10$ initial and terminal constraints.
The results are shown in Table \ref{tab:quadrotor}.
The proposed algorithm is roughly twelve times faster than \pkg{MA57}, the fastest solver for this problem.
The speed-up reduced to roughly seven times when iterative refinement was applied.
The numerical accuracy of the proposed algorithm with iterative refinement was comparable to the accuracy of the other solvers.
\begin{table}[h]\centering
        \begin{tabular}{|c|ccccc|} 
                \hline
                & MUMPS & PARDISO & MA57 & proposed & proposed + it. ref. \\ \hline
                wall time [ms] & 1.21 & 1.47 & 0.97 & 0.075 & 0.14 \\
                $\| \cdot \|_\infty$ residual ratio & $\expnumber{8.90}{-12}$& $\expnumber{6.99}{-11}$ & $\expnumber{1.29}{-12}$ & $\expnumber{1.39}{-09}$  & $\expnumber{1.44}{-13}$ \\ \hline
        \end{tabular}
        \caption{performance and numerical accuracy for the quadrotor problem}\label{tab:quadrotor}
\end{table}
\section{Discussion}
\label{sec:discussion}
\textbf{Comparison with other work.} 
Several connections between the method proposed in this paper and other approaches for solving the constrained LQ problem can be made:
\begin{itemize}
        \item Sideris \& Rodriguez \cite{sideris2010riccati} developed a factorization algorithm that is able to solve the OCP when the initial state is fixed. When ${G_{k,u}}$ has full row rank the computational complexity is linear with the horizon length, but in general it is cubic. Furthermore, the algorithm assumes positive definiteness of the full-space Hessian.
        \item Giftthaler \& Buchli \cite{giftthaler2017projection} also make use of the nullspace of ${\overline{G}_{u}}$.
        Whereas we use a parameterization of a basis of this nullspace, their approach makes use of a projection onto this nullspace.
This results in a (possibly) singular optimal control problem which can be solved by a modified version of the classical Riccati recursion.
The approach has a computational complexity that scales linearly with the horizon length but puts some restrictions on the stagewise constraints,
which for example makes the algorithm unable to solve problems where the number of stagewise equalities exceeds the control input dimensions $n_u$.
In addition, positive semi-definiteness of $Q_k$ and positive definiteness of $R_k$ are assumed.
\item The work of Domahidi et al. \cite{domahidi2012efficient} is used as the linear solver in the \pkg{FORCES NLP} solver \cite{zanelli2020forces}. This solver has proven its use in many application domains such as robotics, automotive and aerospace.
The linear solver supports a more general formulation than the constrained LQ problem treated in this paper.
It makes use of a structure-exploiting range space method, which is linear with horizon length, but it assumes positive definiteness of the full-space Hessian.
\item The method described by Laine \& Tomlin \cite{laine2018efficient} also makes use of the nullspace of ${\overline{G}_{u}}$, in a similar way as our method.
This method however relies on computationally expensive operations that we avoid, such as an SVD for removing linearly dependent constraints.
\end{itemize}
We did not compare our method to the aforementioned state-of-the-art tailored solvers as they assume regularity constraints that we avoid.
Moreover, these papers do not provide publicly available implementations of their algorithms.

\noindent \textbf{Optimal feedback control policies.}
Apart from solving the LQ system, the quantities appearing in the backward substitution (Algorithm \ref{alg:fullalgo}) can also be used for obtaining optimal feedback control policies.
These policies are useful for developing DDP-style algorithms.
The optimal feedback control policies can be constructed from Algorithm \ref{alg:fullalgo}, steps \ref{alg:fullalgo-utildeB} to \ref{alg:fullalgo-u} as:
\begin{equation}
        \mathbf{u}_k = K_k \mathbf{x}_k + \mathbf{k}_k,
\end{equation}
with
\begin{equation}
        K_k = T_{k,R}^{-1} \left[ \begin{array}{c} \widetilde{G}_{k,u} \\ -\Lambda_k \mintranspose L_k \end{array} \right]
\end{equation}
and
\begin{equation}
        k_k  = T_{k,R}^{-1} \left[\begin{array}{c}
                        \widetilde{\mathbf{g}}_k \\-\Lambda_k  \mintranspose \mathbf{l}_k
                \end{array}\right].
\end{equation}
Here, we recognize $K_k$ as the feedback gain matrix and $k_k$ as the feedforward term.
This policy automatically satisfies the subset of the constraints that can be satisfied by choosing the appropriate inputs at the current stage, as explained in Section \ref{sec:recursion}.
For the first stage the following state constraint has to be satisfied:
\begin{equation}
        H_0 \mathbf{x}_0 = -\mathbf{h}_0.
\end{equation}

\noindent \textbf{Limitations of the approach.} 
For large problem state dimensions this resulted in a numerical accuracy that is inferior as compared to the tested general-purpose linear solvers when no iterative refinement is applied.
We believe this is because the algorithm does not implement advanced pivoting strategies.
To tackle this problem more advanced pivoting or scaling strategies could be implemented.
A difficulty is that the approach allows less freedom in the choice of pivots as compared to general-purpose indefinite solvers.
Another possible solution is a different choice of matrix decomposition for the stagewise constraints such as an SVD decomposition instead of the LU decomposition used in this paper.
These problems are left for future work.

\noindent \textbf{Software implementation.} 
The \cpluspluslogo{} implementation of the algorithm proposed in this paper is freely available at \url{https://github.com/lvanroye/generalization_riccati}.
This repository also contains the code for reproducing the results of this paper.

\section{Conclusion and outlook} \label{sec:conclusions}
We developed a novel algorithm for efficiently calculating optimal feedback control policies and solutions of stagewise equality-constrained LQ problems without imposing conservative regularity conditions on the problem at hand.
Numerical experiments illustrate that this approach allows for computationally efficient implementations.
The promising potential of the proposed approach encourages the authors to develop a nonlinear optimal control problem solver, using the discussed recursion as linear solver.
A topic of future work is to investigate more rigorously the impact of inexact arithmetic and improving the numerical accuracy.

\subsection*{FUNDING INFORMATION}
This work was supported by European Research Council (ERC) under the European Union's Horizon 2020 research and innovation programme (Grant agreement: ROBOTGENSKILL No. 788298) and FWO project G0D1119N of the Research Foundation - Flanders (FWO - Flanders).
\subsection*{ORCID}
\textit{Lander Vanroye} \orcidm{0000-0003-3968-1061} \\
\textit{Joris De Schutter} \orcidm{0000-0001-9619-5815} \\
\textit{Wilm Decré} \orcidm{0000-0002-9724-8103}




\bibliography{wileyNJD-AMA}%
\section*{APPENDIX}
\label{appendix:B}
In this appendix we show that the reduced Hessian of the LQ optimal control problem \ref{eq:LQOCP} is positive definite if and only if the matrices $\widehat{R}_k$ for all $k = K-1, \dots, 0$ and $\widehat{P}_I$ appearing in Algorithm \ref{alg:fullalgo} are positive definite. 
The algorithm consists of three different types of steps: substitution, symmetric transformation and Schur-complement steps. Each step transforms the KKT matrix. In the remainder of this appendix we will show that the reduced Hessian
of the KKT matrix before transformation is positive definite if and only if the reduced Hessian of the transformed KKT matrix is positive definite. For a Schur complement step we also require that the Hessian submatrix that is associated with the eliminated variables is positive definite.
\begin{definition}[Inertia of a matrix] \cite{nocedal2006numerical}
        We define the inertia of a symmetric matrix $\mathcal{K}$ as the scalar triple that indicates the number $n_+$, $n_-$ and $n_0$ of positive, negative and zero eigenvalues, respectively: \[\text{inertia}(\mathcal{K})=(n_+, n_-, n_0). \]
\end{definition}
\begin{theorem}[Sylvester's Law of Inertia ]\cite{SylvestersTheorem}\label{Sylvesters}
        Let $A$ and $B$ be congruent real symmetric matrices. This means that $B = S A S\transpose$ for some invertible matrix $S$. Then $A$ and $B$ have the same inertia.
\end{theorem}
\begin{theorem}\label{inertiath}\cite{gould1985practical,nocedal2006numerical}
        Consider the KKT matrix $\mathcal{K}$ \eqref{eq:kdef} and Definition \ref{def:reducedhess}, and suppose $A$ has full row rank. Then  \[\text{inertia}(\mathcal{K})= \text{inertia}\left(Z\transpose HZ\right) + (m,m,0) .\]
\end{theorem}
\begin{corollary}
        A KKT matrix $\mathcal{K}$ \eqref{eq:kdef} has a positive definite reduced Hessian if and only if \[ \text{inertia}(\mathcal{K}) = (n, m ,0). \]
\end{corollary}
\subsection*{Substitution step}
\begin{definition} \label{defsubs}
        A nullspace projection step transforms the KKT matrix \begin{equation} \label{defsubseq}\mathcal{K}_s =  \left[
                        \begin{array}{cc;{2pt/2pt}c}
                                H & A\transpose & B\transpose \\
                                A &             &             \\ \hdashline[2pt/2pt]
                                B &             &             \\
                        \end{array}
                        \right],  \end{equation}
        where $H \in \mathbb{S}^n$, $A\in\mathbb{R}^{m_A \times n}$ and $B\in\mathbb{R}^{m_B \times n}$ and $\text{rank}(B)= m_B$, to a KKT matrix
        \[
                \widehat{\mathcal{K}}_s=\begin{bmatrix}
                        Z_B\transpose H Z_B & Z_B\transpose A\transpose \\ A Z_B &
                \end{bmatrix},
        \]
        where $Z_B$ is a matrix whose columns span the nullspace of $B$.
\end{definition}
In the substitution steps of the proposed algorithm we have $B = \begin{bmatrix} -\identm{\rho} & \zerom{} \end{bmatrix}$. So $Z_B = \begin{bmatrix}\zerom{} & \identm{m_B-\rho} \end{bmatrix}\transpose$. It is easy to verify that substitution of this equation and elimination of the associated dual variables corresponds to the transformation of Definition \ref{defsubs}.
\begin{theorem}
        The reduced Hessian of $\mathcal{K}_s$ \eqref{defsubseq} is positive definite if and only if the reduced Hessian of the transformed KKT matrix $\widehat{\mathcal{K}}_s$ is positive definite.
\end{theorem}
\begin{proof}
        A matrix whose columns span the nullspace of the bottom partition of the constraint Jacobian of \eqref{defsubseq} is $Z_s = \begin{bmatrix} Z_B & \\ &\identm{m_A} \end{bmatrix}$. Because $\widehat{\mathcal{K}}_s = Z_s\transpose \mathcal{K}_s Z_s$ and of Theorem \ref{inertiath}: $\text{inertia}(\mathcal{K}_s) = \text{inertia}(\widehat{\mathcal{K}}_s) + (m_B, m_B, 0)$.
\end{proof}

\addtolength{\textheight}{-7cm}

\subsection*{Schur complement step}
\begin{definition}
        A Schur complement step transforms the KKT matrix
        \begin{equation} \label{defschur}
                \mathcal{K}_c = \begin{bmatrix}
                        H & S\transpose &             \\
                        S & G           & A\transpose \\
                          & A           &
                \end{bmatrix},
        \end{equation}
        where $H \in \mathbb{S}^{n_H}$ and $G\in \mathbb{S}^{n_G}$ and $A \in \mathbb{R}^{m\times n_G}$ and $\text{rank}(A)=m$, to a KKT matrix
        \begin{equation}
                \widehat{\mathcal{K}}_c = \begin{bmatrix}
                        G - SH^{-1}S\transpose & A\transpose \\
                        A                      &
                \end{bmatrix}.
        \end{equation}
\end{definition}
\begin{lemma} \label{Hpos}
        If the reduced Hessian of $\mathcal{K}_c$ in \eqref{defschur} is positive definite, then $H$ is positive definite.
\end{lemma}
\begin{proof}
        The constraint Jacobian of $\mathcal{K}_c$ has $Z = \begin{bmatrix} \identm{n_H} & \\ & Z_A\end{bmatrix}$ as a matrix whose columns span its nullspace, where the columns of $Z_A$ span the nullspace of $A$. We can calculate the reduced Hessian $R_c$ of KKT matrix $\mathcal{K}_c$ as \[R_c = \begin{bmatrix} H & S\transpose Z_A \\ Z_A\transpose S & Z_A\transpose G Z_A \end{bmatrix} .\]
        This matrix can only be positive definite if the principal submatrix $H \in \mathbb{S}^{n_H}$, of $R_C$, is positive definite.
\end{proof}
This also means that in this case $H$ is invertible. This fact is used in the proof of Theorem \ref{theorem:redHessposdef}:
\begin{theorem}
        The reduced Hessian of $\mathcal{K}_c$ \eqref{defschur} has a positive definite reduced Hessian if and only if $H$ and $\widehat{\mathcal{K}}_c$, the reduced Hessian of the transformed KKT matrix, are positive definite.
        \label{theorem:redHessposdef}
\end{theorem}
\begin{proof}
        Define the invertible matrix $S = \begin{bmatrix} \identm{} & & \\ -SH \mintranspose& \identm{} & \\ & & \identm{} \end{bmatrix}$, such that $S\mathcal{K}_c S\transpose= \begin{bmatrix} H & \\ & \widehat{\mathcal{K}}_s \end{bmatrix}$.
        Then by Sylvester's law of inertia\[\text{inertia}(\mathcal{K}_c) = \text{inertia}(H) + \text{inertia}(\widehat{\mathcal{K}}_c). \]
        \\
        $\Rightarrow$ Since the reduced Hessian of $\mathcal{K}_c$ is positive definite, $\text{inertia}(\mathcal{K}_c) = (n_H+n_G, m_A,0)$ and because of Theorem \ref{Hpos}, $\text{inertia}(H) = (n_H, 0,0)$. Such that $\text{inertia}(\widehat{\mathcal{K}}_c) = (n_G, m_A,0)$.
        \\
        $\Leftarrow$  Since $\text{inertia}(H) = (n_H, 0,0)$ and $\text{inertia}(\widehat{\mathcal{K}}_c) = (n_G, m_A,0)$, $\text{inertia}(\mathcal{K}_c) = (n_H+n_G, m_A,0)$.
\end{proof}
\subsection*{Symmetric transformation step}
\begin{definition}
        A symmetric transformation step transforms a KKT system $\mathcal{K}_t$ to a KKT system $\widehat{\mathcal{K}}_t$
        \begin{equation}\label{vartransdef}
                \widehat{\mathcal{K}}_t = M\transpose \mathcal{K}_t M,
        \end{equation}
        with $M \in \mathbb{R}^{(m+n)\times(m+n)}$ invertible.
\end{definition}
\begin{theorem}
        The reduced Hessian of KKT matrix $\mathcal{K}_t$ \eqref{vartransdef} is positive definite if and only if the reduced Hessian of $\widehat{\mathcal{K}}_t$ is positive definite.
\end{theorem}
\begin{proof}
        This result follows directly from Sylvester's law of inertia, Theorem \ref{Sylvesters}, and Theorem \ref{inertiath}.
\end{proof}

\end{document}